\documentclass[11pt, letterpaper,english]{amsart}
\usepackage[left=1in,right=1in,bottom=1in,top=1in]{geometry}
\usepackage{amsmath,amsthm,amssymb}
\usepackage{mathrsfs}
\usepackage{cite}
\usepackage{hyperref}
\usepackage{enumerate}
\usepackage[usenames,dvipsnames]{color}
\usepackage{colonequals} 
\usepackage{quiver}

\theoremstyle{plain}
\newtheorem{lem}{Lemma}[section]
\newtheorem{lemma}[lem]{Lemma}

\newtheorem{theorem}[lem]{Theorem}

\newtheorem{proposition}[lem]{Proposition}

\newtheorem{conj}[lem]{Conjecture}

\theoremstyle{definition}

\newtheorem{definition}[lem]{Definition}
\newtheorem{example}[lem]{Example}
\newtheorem{remark}[lem]{Remark}

\newtheorem{notation}[lem]{Notation}
\newtheorem{rationale}[lem]{Rationale}

\numberwithin{equation}{section}

\newcommand{\mathfont}{\mathbf}

\newcommand{\Z}{\mathfont Z}
\newcommand{\ZZ}{\mathfont Z}
\newcommand{\F}{\mathfont F}
\newcommand{\FF}{\mathfont F}

\newcommand{\CCC}{C}

\newcommand{\TAU}{\tau}

\newcommand{\cO}{\mathcal{O}}
\newcommand{\cM}{\mathcal{M}}
\newcommand{\cA}{\mathcal{A}}
\newcommand{\cR}{\mathcal{R}}

\newcommand{\into}{\hookrightarrow}
 
\DeclareMathOperator{\ord}{ord}

\newcommand{\PP}{\mathfont{P}}

\DeclareMathOperator{\Jac}{Jac}

\newcommand{\floor}[1]{\left \lfloor #1 \right \rfloor}

\title{Producing supersingular curves of genus five}
\author{Jeremy Booher}
\address{University of Florida}
\email{jeremybooher@ufl.edu}
\author{Rachel Pries}
\address{Colorado State University}
\email{pries@colostate.edu}
\date{\today}

\thanks{Pries was supported by NSF grant DMS-22-00418.  We thank Jeff Achter, Dusan Dragutinovi\'c, Everett Howe, and Valentijn Karemaker 
for helpful conversations.}

\begin{document}

\begin{abstract}
For a prime $p$ congruent to three modulo four, 
we prove that there exists a smooth curve of genus five in characteristic $p$ that is supersingular. 
We produce this curve as an unramified double cover of a curve of genus three.  
We conjecture that the setting of unramified double covers of curves of genus three 
also produces supersingular curves of genus five when $p$ is congruent to one modulo four, 
and we computationally verify this conjecture for primes less than $100$.
These results can be viewed as a generalization of work of Ekedahl and of Harashita, Kudo, and Senda.


2020 MSC primary: 11G20, 11M38, 14H30, 14H40. 
Secondary 11G10, 11G18, 14H10.


\end{abstract}

\maketitle

\section{Introduction}

Suppose $k$ is an algebraically closed field of characteristic $p$ where $p$ is a prime number.  
Suppose $A$ is a principally polarized (p.p.) abelian variety of dimension $g$ defined over $k$.
Then $A$ is \emph{supersingular} if the only slope of its Newton polygon is $1/2$.
This is equivalent to $A$ being isogenous to a product of $g$ supersingular elliptic curves, 
by \cite[Theorem~2d]{tate:endo} and \cite[Theorem~4.2]{oortsub}.

Now suppose $X$ is a smooth (projective, irreducible) curve of genus $g$ over $k$. 
The curve $X$ is \emph{supersingular} if its Jacobian is supersingular. 
For each integer $g \geq 1$ and prime $p$, it is natural to ask whether there exists a smooth curve of 
genus $g$ defined over $\overline{\FF}_p$ that is supersingular.
The answer is known to be yes:
\begin{enumerate}
\item when $g=1$ for all $p$, by Deuring \cite{deuring};
\item when $g=2$ for all $p$, by Serre \cite[Th\'eor\`eme 3]{serre82}, also \cite[Proposition~3.1]{iko};
\item when $g=3$ for all $p$, by Oort \cite[Theorem~5.12]{oorthypsup}; also \cite[Theorem~1]{ibukiyama93};
\item and when $g=4$ for all $p$, by Harashita, Kudo, and Senda \cite[Corollary~1.2]{khs20}; 
for an alternative proof, see \cite[Theorem~1.1]{P:somecasesOort}.
\end{enumerate}

When $p=2$, the answer is yes for all $g$ by \cite{vdgvdv}; 
thus we restrict to the case that $p$ is odd.

In this paper, we propose a conjecture about the existence of supersingular curves of genus $5$ in a new setting, 
and we provide evidence for this conjecture.
We
suppose $X$ is a smooth curve of genus $3$ and $\pi: Y \to X$ is an unramified double cover, so $Y$ has genus $5$.
Then $\Jac(Y)$ is isogenous to $\Jac(X) \times P$, where the Prym $P$ of $\pi$ is a p.p.\ abelian surface.
So $Y$ is supersingular if and only if $X$ and $P$ are both supersingular.

\begin{conj} \label{conjecture}
For any odd prime $p$, 
there exists a smooth curve $X$ of genus $3$ over $\overline{\FF}_p$
having an unramified double cover $\pi : Y \to X$ such that
$Y$ is a supersingular curve of genus $5$.
\end{conj}

Here is the main idea behind this conjecture.
Within the moduli space $\cM_3$ of smooth curves of genus $3$, the supersingular locus has dimension $2$.  
Within the moduli space $\cA_2$ of p.p.\ abelian surfaces, the supersingular locus has codimension $2$.
Thus, we might expect to find finitely many supersingular curves $X$ of genus $3$ that 
have an unramified double cover $\pi:Y \to X$ whose Prym is also supersingular, and the 
difficulty is showing that not all of these occur at the boundary of $\cM_3$.

\subsection{Verification when \texorpdfstring{$p \equiv 3 \pmod{4}$}{p=3 mod 4}}
In Section~\ref{Section3}, we provide our main theoretical evidence 
by proving Conjecture~\ref{conjecture} when $p \equiv 3 \bmod 4$.

\begin{theorem} \label{Tintro}
For any prime $p \equiv 3 \bmod 4$, 
there exists a smooth curve $X$ of genus $3$ defined over $\overline{\FF}_p$
having an unramified double cover $\pi : Y \to X$ such that $Y$ is supersingular.
In particular, there exists a smooth curve $Y$ of genus five defined over 
$\overline{\FF}_p$ that is supersingular for any $p \equiv 3 \bmod 4$.
\end{theorem}

The proof of Theorem~\ref{Tintro} shares two features with the result of \cite{khs20} 
(about the existence of supersingular curves of genus $4$ for all primes $p$): both construct supersingular curves as Klein-four covers of the projective line $\mathfont{P}^1$ and both are, in principal, constructive.  However, the techniques are very different.  

Our proof uses a new strategy involving special Shimura varieties.
We study a two-dimensional family $M[8]$ of curves $X$ of genus $3$.
This is the eighth family of cyclic covers of $\mathfont{P}^1$ considered by Moonen \cite{moonen};
the image of the $M[8]$ family under the Torelli morphism is a special subvariety of the moduli space $\cA_3$ of p.p.\ abelian 
threefolds.  
By taking an unramified degree two cover $\pi: Y \to X$, we obtain a two-dimensional family of curves of genus $5$, 
where the Prym $P$ of $\pi$ is a p.p.\ abelian surface.
The $M[8]$ family has an intriguing property: when $p \equiv 3 \bmod 4$, 
the generic curve $X$ has $p$-rank $2$ and the generic Prym $P$ has $p$-rank $1$.  

Using a result of Tamagawa \cite{tamagawa03},
we determine a polynomial condition on $M[8]$ for which the $p$-rank of $X$ drops and another for 
which the $p$-rank of $P$ drops.  By default, when both polynomial conditions are satisfied, 
we would expect that $Y$ has $p$-rank $1$;
instead, the Newton polygons of $X$ and $P$ are actually supersingular, 
and so $Y$ is in fact supersingular!  
To complete the proof, we show that both polynomial conditions are satisfied at a point in the family where $X$ and $Y$ are both smooth.  
This is an important step in the proof, and not a formality: the closure of the $M[8]$ family intersects the boundary of $\mathcal{M}_3$, 
so some curves in the complete family are singular.

\subsection{Computational Evidence}
In Section~\ref{sec:computational}, we provide computational evidence for Conjecture~\ref{conjecture} in the open case 
when $p \equiv 1 \bmod{4}$.

\begin{theorem} \label{Tcomputation}
Conjecture~\ref{conjecture} is also true for all primes $p \equiv 1 \bmod{4}$ which are less than $100$.
\end{theorem}

The strategy for Theorem~\ref{Tcomputation} is to start with a supersingular p.p.\ abelian surface $P$;
the LMFDB provides a list of these over $\mathfont{F}_p$ for $p$ in this range \cite{lmfdb}.
Then we find smooth quartic plane curves $X$ having an unramified double cover $\pi : Y \to X$ whose Prym is $P$.
By \cite{Beauville} and \cite{Verra}, such curves $X$ occur as the intersection of a plane $V$ with a 
projective model of the Kummer surface $K$ of $P$ in $\mathfont{P}^3$.
We computationally verify the conjecture for $p < 100$ by searching for planes $V$ 
for which $X = V \cap K$ is a smooth curve of genus $3$, 
not containing any of the singularities of $K$, and which is supersingular.

\subsection{Rationale for the Conjecture}

In Section~\ref{Srationale}, we provide several rationales for Conjecture~\ref{conjecture}.
The basic idea is to compare the dimension of the supersingular locus of the moduli space $\cM_g$ of curves of genus $g$
with the codimension of the supersingular locus in $\cA_{g-1}$.  Conjecture~\ref{conjecture} is based on this comparison when $g=3$.
This comparison also works well when $g=2$ and, surprisingly, does not seem to have been investigated before.
For $p \equiv 3 \bmod 4$, we include a result about supersingular curves of genus $3$ that are unramified double covers of a genus $2$ curve. 

As a variation, we consider supersingular curves that are double covers of another curve branched at exactly two points.
This includes the supersingular Howe curves of genus $4$ found in \cite{khs20}.
In Conjecture~\ref{conjecture2}, we propose a variation of Conjecture~\ref{conjecture} about supersingular curves of genus $6$ that are double covers of a genus $3$ curve.

In Section~\ref{Srationale}, we also explain the difficulties in proving these conjectures that are caused by (families of) supersingular singular curves.

\subsection{Relation to Previous Work}
After proving Theorem~\ref{Tintro}, 
we searched the literature for earlier results about supersingular curves in genus $5$ in characteristic $p$.
We found several such results, which apply under a patchwork of
congruence conditions on $p$, but which ultimately do not treat all primes with $p \equiv 3 \bmod{4}$.

These earlier results involve abelian covers $\TAU: C \to \mathfont{P}^1$ branched at three points.\footnote{Abelian covers of $\mathfont{P}^1$ branched at three points are quotients of Fermat curves.
Many authors determined conditions when such curves are supersingular \cite{Yui,Aoki,re2008,re2009}.
In \cite[Theorem~13]{re2009}, Re found a supersingular curve of genus $g$ over $\overline{\FF}_p$ 
for all but 16 pairs $(g,p)$ such that $1 \leq g \leq 100$ and $3 \leq p \leq 23$.}  
Such a curve $C$ is supersingular when $p$ satisfies appropriate congruence conditions modulo $\mathrm{deg}(\TAU)$.  
This situation produces a supersingular curve of genus $5$ over $\overline{\FF}_p$ if and only if at 
least one of the following holds:
$p \equiv -1 \bmod 8,11,12,15,20$ or $p \equiv -4 \bmod 15$ \cite[page 173]{Ekedahl87};
$p$ is a quadratic non-residue modulo $11$ \cite{LMPT1}; or $p \equiv 11 \bmod 20$ \cite[Theorem~6.1]{BPrimsproc}.

These cases do not cover all primes $p \equiv 3 \bmod 4$, with the new cases being $p \equiv 43,67 \bmod 120$.
Thus Theorem~\ref{Tintro} verifies the existence of supersingular genus $5$ curves over $\overline{\FF}_p$ for infinitely many more primes $p$, 
using a new approach which applies uniformly for all $p \equiv 3 \bmod 4$. 
Furthermore, these earlier constructions do not verify Conjecture~\ref{conjecture} because none of these supersingular 
curves of genus $5$ have fixed-point free involutions.

For primes $p$ with $p \equiv 1 \bmod{4}$, the smallest primes for which the earlier results do not apply are 
$37$, $53$, and $97$, which are covered by Theorem~\ref{Tcomputation}.


\section{A Criterion for Supersingularity} \label{Section2}

Throughout, we work over an algebraically closed field $k$ whose characteristic $p$ is an odd prime.
All curves in Sections~\ref{Section2}, \ref{Section3}, and \ref{sec:computational} are smooth projective irreducible curves over $k$.

Let $\TAU : \CCC \to W$ be a branched cyclic cover of smooth curves.  
Let $m$ be the degree of $\TAU$ and suppose that $p \nmid m$.
In this section, we describe a technique for obtaining information about the $p$-rank of $\CCC$.
The approach is based on work of Tamagawa about generalized Hasse--Witt invariants \cite[Section 3]{tamagawa03}.

\subsection{The Group Action and Frobenius}\label{ss:tamagawa}
The $\Z/m\Z$-action on $\CCC$ induces a $\Z/m \Z$-action on $\cO_\CCC$ and on $H^1(\CCC,\cO_\CCC)$.  Fix a primitive character $\chi : \Z/m\Z \to k^\times$.  As in \cite[Section 3]{tamagawa03}, $\TAU_* \cO_\CCC$ decomposes as a direct sum of line bundles on $W$:
\begin{equation} \label{eq:decomposepushforward}
\TAU_* \cO_\CCC \cong \bigoplus_{i \in \Z/m\Z} L_i,
\end{equation}
where $\Z/m\Z$ acts on $L_i$ via $\chi^i : \Z/m\Z \to k^\times$.

The branched cyclic cover $\TAU : \CCC \to W$ can be described using geometric class field theory 
as the pullback of a cyclic cover of a generalized Jacobian of $W$.  
Let $S \subset W(k)$ be the set of branch points of $\TAU$.  Let $U \colonequals W - S$ and
$V \colonequals \CCC-\TAU^{-1}(S)$.  
Then $\TAU|_V : V \to U$ is \'{e}tale, and $V$ is a $\mu_m$-torsor over $U$.  
By \cite[Proposition~3.5]{tamagawa03}, the cover $\TAU : \CCC \to W$ corresponds to a line bundle $L$ on $W$ and a divisor $D$ supported on $S$ with coefficients in $\{0,1,\ldots,m-1\}$ such that $L^{\otimes m} \otimes \cO_W(D)$ is trivial.  Furthermore, the proof shows that we may assume $L = L_1$.

For a divisor $D'$ on $W$, we set $\floor{D'/m} \colonequals \sum_{P \in W} \floor{\ord_{P}(D')/m} P$ where $\floor{x}$ denotes the greatest integer less than or equal to $x$.

\begin{proposition} \label{prop:tamagawa line bundles}
The line bundle $L_i$ appearing in \eqref{eq:decomposepushforward} equals $L^{\otimes i}(\floor{i D / m})$.  
\end{proposition}

\begin{proof}
This is established in the proof of \cite[Claim~(3.8)]{tamagawa03}.
\end{proof}

We wish to study the Frobenius map on 
\begin{equation}
H^1(\CCC,\cO_\CCC) \cong \bigoplus_{i \in \Z/m\Z} H^1(W,L_i).
\end{equation}
The Frobenius map sends $H^1(W,L_i)$ to $H^1(W,L_{pi})$, where the subscript is taken modulo $m$.  Therefore it is natural to study its $r$th power, where $r$ is the order of $p$ in $(\Z/m\Z)^\times$.  
This is a semi-linear map $\varphi_{\TAU,i} : H^1(W,L_i) \to H^1(W,L_i)$.  
Let $\varphi \colonequals \varphi_{\TAU,1}$.
The next result provides an explicit description of $\varphi$; 
a similar description is available for $\varphi_{\TAU,i}$ when $i \not = 1$.

\begin{proposition} \label{proposition:frobenius formula}
For a branched cover $\TAU : \CCC \to W$ corresponding to $(L,D)$ as above, fix an isomorphism $\iota: L^{\otimes m} \cong \cO_W(-D)$.  Then $\varphi : H^1(W,L) \to H^1(W,L)$ is induced by the composition
\begin{equation} \label{Elinebundle}
L \overset{F^r} \to L^{\otimes p^r} = L \otimes L^{\otimes (p^r-1)} \overset{\iota} 
\cong L \otimes \cO_W \left( - \left(\frac{p^r-1}{m}\right) D \right) \into L.
\end{equation}
\end{proposition}

\begin{proof}
See \cite[\S 3, page~76]{tamagawa03}.
\end{proof}

\begin{example} \label{example:elliptic curve}
Let $p$ be odd. 
Let $E: y^2 = x(x-1)(x-\lambda)$ be an elliptic curve over $k$.
As in \cite[Theorem V.4.1(b)]{silvermanaec}, let
\begin{equation} \label{EHplambda}
H_p(\lambda) \colonequals 
\sum_{j=0}^{(p-1)/2} \binom{(p-1)/2}{j}^2 \lambda^j. 
\end{equation}
It is well-known that $E$ is supersingular if and only if $H_p(\lambda)=0$.
We explain how this follows from Proposition~\ref{proposition:frobenius formula}, taking $\CCC=E$ and 
$W= \PP^1$ and $\TAU : \CCC \to W$ the projection map onto the $x$-axis. 

In this case, $m=2$ and $S= \{0,1,\lambda,\infty\}$.
The cover $\TAU$ is determined by the line bundle $L = \cO_{\PP^1}(-2 [\infty])$ and the divisor 
$D = [0] + [1] + [\lambda] + [\infty]$.  
The isomorphism $L^{\otimes 2} \cong \cO_{\PP^1}(-D)$ is given by multiplying by the function $x(x-1)(x-\lambda)$.
Then $L_1 = \cO_{\PP^1}(-2 [\infty])$, with $\dim H^1(\PP^1,L_1) = 1$.  
In the standard \v{C}ech description, the equivalence class of $x^{-1}$ is a basis element of $H^1(\PP^1,L_1)$.  
So
\[
\varphi(x^{-1}) = x^{-p} \left( x (x-1)(x-\lambda) \right)^{(p-1)/2} \equiv h_p(\lambda) \bmod{ x^{-2} k[x^{-1}] \oplus k[x]},
\]
where $h_p(\lambda)$ is the coefficient of $x^{-1}$ in $x^{-p} \left( x (x-1)(x-\lambda) \right)^{(p-1)/2}$.  
Note $\varphi$ is the zero map if and only if $h_p(\lambda)=0$, so this condition on $\lambda$ 
is equivalent to $E$ being supersingular.  
By the binomial theorem, one can check that $h_p(\lambda) = \pm H_p(\lambda)$.
\end{example}

\subsection{A Special Family of Curves}  

We now consider a family of covers $\TAU : C_1 \to \mathfont{P}^1$ such that there are restrictions on 
the Newton polygons of the curves $C_1$ in the family.  

Consider the family $M[8]$ of genus $3$ curves that are $\ZZ/4\ZZ$-covers of $\PP^1$ branched at 5 points 
with inertia type $a=(1,1,2,2,2)$.  (The notation $M[8]$ is based on \cite[Table~2]{moonen}.)
Each curve in this family has an affine equation of the form:
\begin{equation} \label{EM8family}
C_1:=C_1(t_1,t_2): y^4=x^2(x-1)^2(x-t_1)(x-t_2),
\end{equation}
for some $t_1, t_2 \in k - \{0,1\}$ with $t_1 \not = t_2$.  The group $\ZZ/4 \ZZ \cong \mu_4$ naturally acts on $C_1$ via 
multiplication on $y$, and projecting to the $x$-coordinate gives a $\Z/4\Z$-cover $\TAU : C_1 \to \PP^1$.  

\begin{proposition} \label{cor:possible prank}
Let $p \equiv 3 \bmod 4$.  For the generic choice of $t_1, t_2$, 
the $p$-rank of $C_1(t_1,t_2)$ is $2$.  If its $p$-rank is less than $2$, then $C_1(t_1,t_2)$ is supersingular.  
\end{proposition}

\begin{proof}
This is proved in \cite[Section~6.1]{LMPT2}.
The main idea is that the $\ZZ/4\ZZ$-cover $\TAU: C_1 \to \PP^1$ has signature $(f_1,f_2,f_3) = (2,0,1)$.  
Here $f_i=\mathrm{dim}(L'_i)$ where $L'_i$ is the $i$th eigenspace for the $\ZZ/4\ZZ$-action on $H^0(C_1, \Omega^1)$.
By the Kottwitz method, there are constraints on the Newton polygon of a p.p.\ abelian threefold with a
$\ZZ/4\ZZ$-action with this signature.  In particular, there are only two choices for the Newton polygon, called the
$\mu$-ordinary and basic Newton polygons.
In this case, these can be distinguished by the $p$-rank.  For $p \equiv 3 \bmod 4$, the $\mu$-ordinary one has $p$-rank $2$ and the basic 
one is supersingular.
\end{proof}

We make this explicit by determining conditions on $t_1, t_2$ for $C_1(t_1,t_2)$ to be supersingular.

\begin{definition} \label{Dpolyb}
Let $A \colonequals (p^2-1)/4$, and let $b_p(t_1,t_2) \in k[t_1,t_2]$ be the coefficient of $x^{2A}$ in 
\begin{equation}\label{EdefBp}
\left( (x-1)^2 (x-t_1)(x-t_2) \right)^A.
\end{equation}
\end{definition}

\begin{proposition} \label{prop:m8ss}
Suppose $p \equiv 3 \bmod 4$.  The curve $C_1(t_1,t_2)$ in \eqref{EM8family} is supersingular if and only if $b_p(t_1,t_2) =0$.
\end{proposition}

\begin{proof}
We apply the technique from Section~\ref{ss:tamagawa} to the cover $\TAU: C_1 \to \PP^1$.  In that notation, $W = \PP^1$, $m=4$, and $S = \{0,1,t_1,t_2,\infty\}$.  The cover corresponds to $L = \cO_{\PP^1}(-2[\infty])$ and $D = 2[0] + 2[1] + 2 [\infty] + [t_1]+[t_2]$.  The isomorphism $L^{\otimes 4} \cong \cO_{\PP^1}(-D)$ is given by multiplying by the function $x^2(x-1)^2(x-t_1)(x-t_2)$.
We decompose $\TAU_* \cO_{C_1}$ as a direct sum of four line bundles based on the $\Z/4\Z$-action.  
By Proposition~\ref{prop:tamagawa line bundles}, we see that
\begin{align*}
L_0 = \cO_{\PP^1}, \quad L_1 & \cong L = \cO_{\PP^1}(-2[\infty]), \quad L_2 \cong L^{\otimes 2} ([0] + [1] + [\infty]) \cong \cO_{\PP^1}(-[\infty]) \quad \text{and}
\\ L_3 &\cong L^{\otimes 3}([0] + [1] + [\infty]) \cong \cO_{\PP^1}(-3 [\infty]).
\end{align*}
Thus $\dim_k H^1(\PP^1,L_i)$ is $0,1,0,2$ for $i=0,1,2,3$, respectively.  

The Frobenius map permutes these subspaces of $H^1(C_1,\cO_{C_1})$. 
When $p \equiv 3 \bmod{4}$, it exchanges $H^1(\PP^1,L_1)$ and $H^1(\PP^1,L_3)$.  
The $p$-rank is the stable rank of the Hasse--Witt matrix for the Frobenius map. 
The stable rank restricted to $H^1(\PP^1,L_3)$ is at most $1$ since $H^1(\PP^1,L_1)$ is $1$-dimensional.
Consider the composition
\[
\varphi: H^1(\PP^1,L_1) \to H^1(\PP^1,L_3) \to H^1(\PP^1,L_1).
\]

{\bf Claim:} $C_1(t_1,t_2)$ is supersingular if and only if $\varphi$ is the zero map.

\emph{Proof of claim:}
If $\varphi$ is the zero map, then the $p$-rank of $C_1$ is at most one, and so $C_1$ is supersingular by Proposition~\ref{cor:possible prank}.  
On the other hand, if $\varphi$ is non-zero, then $\varphi_{\TAU, 3}$ is also non-zero;
in this case, the $p$-rank (namely, the stable rank of the Hasse--Witt matrix) is two.
This completes the proof of the claim.  

We use Proposition~\ref{proposition:frobenius formula} to see the effect of the parameters $t_1$ and $t_2$ on $\varphi$.  
Recall that $\varphi$ is induced by the map of line bundles $\varphi': L \to L$ from \eqref{Elinebundle}.
Using the isomorphism $L^{\otimes 4} \cong \cO_{\PP^1}(-D)$ given above, and letting $A = (p^2-1)/4$, then 
on local sections $\varphi'$ sends 
\begin{equation} \label{Emapf}
f \mapsto f^{p^2} \left( x^2(x-1)^2(x-t_1)(x-t_2) \right)^{A}.
\end{equation}

In the \u{C}ech description, $H^1(\PP^1,\cO_{\PP^1}(-2))$
consists of the ring of functions $k[x,x^{-1}]$, which are regular except at $0$ and infinity, 
modulo the functions in $k[x]$ (which are regular except at infinity), 
and modulo the functions in $x^{-2} k[x^{-1}]$ (which are regular except at $0$ and have at least a double zero at infinity).
In this quotient, the equivalence class of $x^{-1}$ is non-zero and hence a basis element.  
Thus, by \eqref{Emapf}, $\varphi(x^{-1})$ is the equivalence class represented by $B_p(t_1,t_2) x^{-1}$,
where $B_p(t_1,t_2)$ is the coefficient of $x^{-1}$ in 
$x^{-p^2} \left( x^2(x-1)^2(x-t_1)(x-t_2) \right)^{A}$. 
Simplifying, $B_p(t_1,t_2)$ equals the coefficient of $x^{(p^2-1)/2} = x^{2A}$ in \eqref{EdefBp}.
Thus $C_1(t_1,t_2)$ is supersingular if and only if $\varphi$ is the zero map if and only if $b_p(t_1,t_2)=0$.
\end{proof}

\begin{remark}
Proposition~\ref{prop:m8ss} can also be proven by describing the Cartier operator on $H^0(C_1,\Omega^1_{C_1})$
using \cite{Elkin}, although the computations would be more complicated with that approach.
\end{remark}

\section{Construction of the supersingular curve} \label{Section3}

We construct a supersingular curve of genus five as an unramified double cover of a supersingular curve in the $M[8]$ family when 
$p \equiv 3 \bmod{4}$.  This requires a specific choice of parameters.

\subsection{Constructing Double Covers}

Let $p$ be an odd prime.  Let $X=C_1$ be the curve in \eqref{EM8family} and 
let $\TAU : C_1 \to \PP^1$ be the projection map onto the $x$-coordinate.  

\begin{notation} \label{Nsetup}
Let $h_1 : D_1 \to \PP^1$ be the cover given by $w^2 = x$, with $h_1$ being projection onto the $x$-coordinate.
Let $\pi : Y \to C_1$ be the pullback of $h_1$ by $\TAU$.
\end{notation}

Note that $D_1$ has genus $0$ and $h_1$ is ramified over $0$ and infinity.  
We write $Y(t_1,t_2)$ for $Y$ when we want to denote the dependence of $Y$ on the parameters $t_1$ and $t_2$, 

\begin{lemma} \label{Lconstructcover}
The cover $\pi: Y \to C_1$ is an unramified double cover, and $Y$ has genus $5$.
\end{lemma}

\begin{proof}
The first statement follows from Abhyankar's Lemma, because the covers $h_1$ and $\TAU$ 
both have inertia groups
of order two above $x=0$, and above $x=\infty$.
The second statement follows from the Riemann--Hurwitz formula.
\end{proof}

\begin{notation}
Let $E_1$ and $E_2$ be the elliptic curves given by the following Weierstrass equations:
\begin{equation} \label{EdefineE1E2}
E_1:y^2=x(x-t_1)(x-t_2) \mbox{ and } E_2:y^2 =x(x-r)(x+r) \mbox{ where } r = \sqrt{(1-t_1)/(1-t_2)}.
\end{equation}
\end{notation}

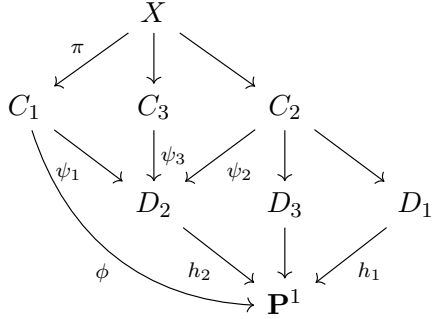
\begin{figure}[th]
\[\begin{tikzcd}
	& Y \\
	{C_1} & {C_3} & {C_2} \\
	& {D_2} & {D_3} & {D_1} \\
	&& {\mathfont{P}^1}
	\arrow["\pi"', from=1-2, to=2-1]
	\arrow[from=1-2, to=2-2]
	\arrow[from=1-2, to=2-3]
	\arrow["{\psi_1}"', from=2-1, to=3-2]
	\arrow["\TAU"', curve={height=30pt}, from=2-1, to=4-3]
	\arrow["{\psi_3}", from=2-2, to=3-2]
	\arrow["{\psi_2}", from=2-3, to=3-2]
	\arrow[from=2-3, to=3-3]
	\arrow[from=2-3, to=3-4]
	\arrow["{h_2}"', from=3-2, to=4-3]
	\arrow[from=3-3, to=4-3]
	\arrow["{h_1}", from=3-4, to=4-3]
\end{tikzcd}\]
\caption{The curves and morphisms appearing in the proof of Proposition~\ref{proposition:prym description}}
\label{curvediagram}
\end{figure}

\begin{proposition} \label{proposition:prym description}
With notation as in \eqref{EdefineE1E2}, 
there is an isogeny $\mathrm{Jac}(Y) \sim \mathrm{Jac}(C_1) \times E_1 \times E_2$.  
In particular, the Prym of $\pi : Y \to C_1$ is isogenous to $E_1 \times E_2$.  
The degree of each of these isogenies is a power of $2$.
\end{proposition}

\begin{proof}
The reader may find it helpful to reference Figure~\ref{curvediagram}.  
The $\Z/4\Z$-cover $\TAU$ factors as a composition of degree two covers $\psi_1 : C_1 \to D_2$ and $h_2:  D_2 \to \PP^1$.
Above each of $x=0,1,\infty$, the curve $C_1$ has two points, with inertia group of order $2$.
Above each of $x=t_1, t_2$, the curve $C_1$ has one point, with inertia group of order $4$.  
Thus $h_2$ is ramified only over $t_1$ and $t_2$, so $D_2$ has genus $0$.  

Consider the fiber product of $h_1 : D_1 \to \PP^1$ and $h_2 : D_2 \to \PP^1$.  
It is a Klein-$4$ cover $C_2 \to \PP^1$, and factors through a third double cover $D_3 \to \PP^1$.  Note that this third cover is ramified over $0,\infty,t_1,t_2$, the set of points over which exactly one of $h_1$ and $h_2$ is ramified.  Thus $D_3$ is isomorphic to the elliptic curve $E_1$.  The Kani-Rosen theorem \cite[Theorem~B]{kanirosen} shows that $\Jac(C_2)$ is isogenous to $E_1$, 
since $D_1$ and $D_2$ have genus $0$.

Let $\psi_2 : C_2 \to D_2$ be the pullback of $h_1$ by $h_2$.  The fiber product of $\psi_1 : C_1 \to D_2$ and $\psi_2 : C_2 \to D_2$ is a Klein-$4$ cover $Y \to D_2$.  It factors through a third double cover $\psi_3 : C_3 \to D_2$.  Note that $\psi_1$ is ramified over eight points: the unique point of $D_2$ above each of $t_1, t_2 \in \PP^1$ and the two points of $D_2$ above each of $0, 1 , \infty \in \PP^1$.  Furthermore, $\psi_2$ is ramified at four points: the two points of $D_2$ above $0$ and $\infty$.  Thus $\psi_3$ is branched over four points of $D_2$: the two points above $1 \in \PP^1$ and the one point above each of $t_1, t_2 \in \PP^1$.   Thus $C_3$ has genus one.  By the Kani-Rosen theorem \cite[Theorem~B]{kanirosen}, we conclude that
$\Jac(Y) \sim \Jac(C_1) \times \Jac(C_2) \times \Jac(C_3)$.

To find an explicit equation for $C_3$, we note that an affine equation for $h_2 : D_2 \to \PP^1$ is 
$z^2 = (x-t_1)(x-t_2)^{-1}$.  The function $z$ identifies $D_2$ with $\mathfont{P}^1$, and the points of $D_2$ above $t_1,t_2,1 \in \PP^1$ correspond to $z = 0, \infty, \pm r$, respectively.  Thus $C_3$ is isomorphic to $E_2$.  Thus
\[
\Jac(Y) \sim \Jac(C_1) \times E_1 \times E_2.
\]

The final statement follows because each decomposition in the Kani-Rosen theorem 
arises from an idempotent relation and the constants in the idempotent relation are powers of $2$.
\end{proof}

\subsection{The supersingular locus}

\begin{proposition} \label{prop:supersingularequivalent}
Let $p \equiv 3 \bmod 4$.  Recall Notation~\ref{Nsetup}.  Then the following are equivalent:
\begin{enumerate}[(i)]
\item  \label{i} the genus $5$ curve $Y(t_1,t_2)$ is supersingular;
\item  \label{ii} the genus $3$ curve $C_1(t_1,t_2)$ and the Prym of $\pi$ are both supersingular;
\item  \label{iii} the curves $C_1(t_1,t_2)$ and $E_1(t_1,t_2)$ are both supersingular; and
\item  \label{iv} $H_p(t_2/t_1)=0$ and $b_p(t_1,t_2)=0$, with $H_p(\lambda)$ as in \eqref{EHplambda} and $b_p(t_1,t_2)$ as in Definition~\ref{Dpolyb}.
\end{enumerate}
\end{proposition}

\begin{proof}
Proposition~\ref{proposition:prym description} shows that (\ref{i}) and (\ref{ii}) are equivalent.  We observe that $E_2$ is isomorphic to the elliptic curve given by $y^2 = x^3-x$, which is supersingular as $p \equiv 3 \bmod 4$, so (\ref{ii}) and (\ref{iii}) are equivalent.  Proposition~\ref{prop:m8ss} shows that $C_1$ is supersingular if and only if $b_p(t_1,t_2)=0$.  Writing $E_1$ in Legendre form, as $y^2 = x(x-1)(x-t_2/t_1)$, and using Example~\ref{example:elliptic curve} completes the proof.
\end{proof}

\begin{remark}
Using the techniques of Section~\ref{ss:tamagawa}, one can show 
that the $p$-rank of the Prym of $\pi$ is $0$ if and only if $c_p(t_1,t_2)=0$, 
where $c_p(t_1,t_2)$ is the coefficient of $x^{(p-1)/2}$ in $((x-t_1)(x-t_2))^{(p-1)/2}$.  
Simplifying shows that $c_p(t_1,t_2)= 0$ if and only if $H_p(t_2/t_1)=0$.  
\end{remark}

For the rest of the section, 
we investigate whether condition (iv) of Proposition~\ref{prop:supersingularequivalent}
is satisfied for a choice of $(t_1,t_2)$ such that $t_1 \neq t_2$ 
and $t_1,t_2 \neq 0,1$.  The curves $X=C_1(t_1,t_2)$ and $Y(t_1,t_2)$ are smooth if and only if these conditions are met.

\begin{example}
Let $p=23$.  The intersection of $b_p(t_1,t_2)=0$ and $t_1^{(p-1)/2} H_p(t_2/t_1) = c_p(t_1,t_2)=0$ in the $(t_1,t_2)$ plane 
contains the following $\FF_{p}$-points:
\begin{equation*}
\{ (5, 19),  (10, 7), (20, 13), (17, 14), (16, 15), 
 (13, 20), (19, 5), 
(15, 16), (7, 10), (14, 17), (1, 22), (22,1) \}
\end{equation*}
For all pairs $(t_1, t_2)$ except the last two, there is an unramified double cover $Y \to C_1(t_1,t_2)$ 
such that $Y$ is a smooth supersingular curve of genus $5$. 
We note that the intersection of the varieties $V(b_p)$ and $V(c_p)$ may not be transverse.  
For example, we computed using Magma \cite{magma} that the local intersection multiplicity at $(t_1,t_2) =(1,22)$ is $6$.
\end{example}

\begin{lemma}
The polynomial $b_p(t_1,t_2)$ is symmetric under the transposition of $t_1$ and $t_2$ 
and has bidegree $(A,A)$, where $A = (p^2-1)/4$. 
Furthermore, 
\begin{equation} \label{Eforbp}
b_p(t_1,t_2)  = \sum_{u=0}^{2A} \sum_{(s_1,s_2)} \binom{2A}{u} \binom{A}{s_1} \binom{A}{s_2} t_1^{A-s_1} t_2^{A-s_2},
\end{equation}
where the inner sum is over pairs $(s_1,s_2)$ such that $s_1, s_2 \geq 0$ and $s_1+s_2 = 2A-u$. 
In particular, $b_p(t_1,t_2)$ has leading term $(t_1t_2)^A$ and constant term $1$.
\end{lemma}

\begin{proof}
By Definition~\ref{Dpolyb},
$b_p(t_1,t_2)$ is the coefficient of $x^{2A}$ in $(x-1)^{2A} (x-t_1)^{A} (x-t_2)^A$.  This is visibly symmetric in $t_1$ and $t_2$.  The formula in \eqref{Eforbp} follows 
from the binomial theorem, after collecting terms involving $x^{2A}$. 
Such terms arise as a product of a term involving $x^u$ from $(x-1)^{2A}$, a term involving $x^{s_1}$ from $(x-t_1)^A$, and a term involving $x^{s_2}$ from $(x-t_2)^A$.
The highest possible exponents on $t_1$ and $t_2$ occur when $s_1 = s_2=0$, in which case $u=2A$ and the leading term is 
$t_1^A t_2^A$.  The constant term occurs when $A = s_1 = s_2$ (so $u=0$), giving constant term $1$.
\end{proof}

\begin{lemma} \label{lemma:F(t)}
The polynomial $B(t) \colonequals b_p(t,-t)$ is an even polynomial of degree $(p^2-1)/2$ and has non-zero constant term. 
The coefficient of $t^2$ in $B(t)$ is congruent to $3/32$ modulo $p$. 
\end{lemma}

\begin{proof}
As $b_p(t_1,t_2)$ is symmetric, it follows that $B(t)$ is an even function of $t$.  As the constant term of $b_p(t_1,t_2)$ is $1$, the constant term of $B(t)$ is also $1$.  The claim about the degree of $B(t)$ follows from the fact that $b_p(t_1,t_2)$ has leading term $(t_1 t_2)^{A}$.

Let $\delta$ be the coefficient of $t^2$ in $B(t) = b_p(t,-t)$.
The terms in the sum in \eqref{Eforbp} that contribute to $\delta$ occur when $u=2$ and $(s_1,s_2)$ is one of $(A-2,A)$, $(A-1,A-1)$, or $(A,A-2)$.  Thus
\begin{eqnarray*}
\delta & = &
\binom{2A}{2} \left( \binom{A}{A-2} \binom{A}{A} - {\binom{A}{A-1}}^2 + \binom{A}{A-2} \binom{A}{A} \right)\\
& = & A (2A-1) (A(A-1) - A^2)= - A^2 (2A-1),
\end{eqnarray*}
where $A = (p^2-1)/4 \equiv - 1/4 \bmod p$.  Simplifying the coefficient modulo $p$ gives the result.
\end{proof}

\begin{proposition} \label{prop:nontrivialsolution}
Let $p \equiv 3 \bmod 4$.  There exist $t_1,t_2 \in \overline{\FF}_p$ with $t_1, t_2 \neq 0, 1$ and $t_1 \neq t_2$ such that $b_p(t_1,t_2)=0$ and $H_p(t_2/t_1) =0$. 
\end{proposition}

\begin{proof}
We restrict to a choice of parameters where $t_1 = -t_2$ because, when $p \equiv 3 \bmod 4$, the elliptic curve $y^2 = x (x-t)(x+t)$ is supersingular, which is reflected in the fact that $-1$ is a root of $H_p(\lambda)$.  
It thus suffices to find $t \in \overline{\F}_p$ such that $B(t):=b_p(t,-t)=0$ and $t \neq 0, \pm 1$.

By Lemma~\ref{lemma:F(t)}, $B(t)$ has $(p^2-1)/2$ roots counted with multiplicity.  
Because its constant term is non-zero, $B(t)$ does not have a root when $t=0$.  

Assume that the only roots of $B(t)$ are $1$ and $-1$. 
Since $B(t)$ is even of degree $(p^2-1)/2$, this implies that $B(t) = (1-t^2)^{(p^2-1)/4}$.  
Then the coefficient of $t^2$ in $B(t)$ is $- \frac{p^2-1}{4} \equiv \frac{1}{4} \bmod p$.  
But $\frac{1}{4} \not \equiv \frac{3}{32} \bmod p$ unless $p=5$, and by hypothesis $p \equiv 3 \bmod{4}$.  
This contradicts 
Lemma~\ref{lemma:F(t)}, so $B(t)$ has a root $t_\circ$ other than $\pm 1$.  
Taking $(t_1,t_2) = (t_\circ,-t_\circ)$ completes the proof.
\end{proof}

We can now prove our main result.

\begin{proof}[Proof of Theorem~\ref{Tintro}]
Let $p \equiv 3 \bmod 4$.  By Proposition~\ref{prop:nontrivialsolution}, there exist $t_1,t_2 \in \overline{\FF}_p$ with $t_1, t_2 \neq 0, 1$ and $t_1 \neq t_2$ such that $b_p(t_1,t_2)=0$ and $H_p(t_2/t_1) =0$. 
Under these restrictions on $t_1$ and $t_2$, the curve $C_1(t_1,t_2)$ with affine equation in \eqref{EM8family} is a smooth 
projective connected curve over $\bar{\mathfont{F}}_p$ of genus $3$.
By Lemma~\ref{Lconstructcover}, there is an unramified double cover $\pi: Y(t_1,t_2) \to C_1(t_1,t_2)$, and so $Y(t_1,t_2)$ 
is a smooth projective curve of genus $5$.
By Proposition~\ref{prop:supersingularequivalent}, $Y(t_1,t_2)$ is supersingular.
\end{proof}

\begin{remark}
The $a$-number of a curve $C$ is $\mathrm{dim}_k \mathrm{Hom}(\alpha_p, \Jac(C))$,
where $\alpha_p$ is the kernel of Frobenius on $\mathfont{G}_a$.
The $a$-number of the supersingular curve $Y$ in Theorem~\ref{Tintro} is at least $3$, because the degree of
the isogeny in Proposition~\ref{proposition:prym description} is a power of $2$ and $\Jac(C_1)$, $E_1$, and $E_2$ each have
$a$-number at least $1$.
\end{remark}
 
\section{Computational Results} \label{sec:computational}

In this section, we provide computational evidence for Conjecture~\ref{conjecture} in the form of Theorem~\ref{Tcomputation}.  
The strategy is to choose a supersingular abelian surface $P$, 
to study smooth quartic plane curves $X$ of genus $3$ that have an unramified double cover $\pi : Y \to X$ whose Prym is $P$, 
and to search for such a curve $X$ which is supersingular.  

\subsection{Constructing Covers with Specified Prym} \label{SspecifiedPrym}

We learned this material from \cite[Section~7]{Bruin}.  Let $P$ be a p.p.\ abelian surface over $k$.  
Let $K$ be its Kummer surface, which is the quotient of $P$ by $[-1]$. 
 Let $\phi : P \to K$ be the degree two quotient map.  
Then $K$ can be embedded as a quartic surface in $\PP^3$, with $16$ singularities which are the images under $\phi$ of the 
$2$-torsion points of $P$.  Note that $\phi$ is ramified over these $16$ points and unramified elsewhere.

For a general plane $V$ in $\PP^3$, the intersection $X = K \cap V$ is a smooth quartic plane curve
of genus $3$.  As long as $X$ does not contain any of the singularities of $K$, then the 
restriction of $\phi$ to $Y := \phi^{-1}(X)$ is an unramified double cover $\pi:Y \to X$. 
Now $P$ is the Prym variety of $\pi$ by \cite[page~616]{Beauville}, so $\Jac(Y)$ is isogenous to $\Jac(X) \times P$.
In fact, every smooth quartic plane curve $X$ having an unramified double cover 
$\pi:Y \to X$ with Prym $P$ arises by this construction \cite[Corollary~4.1]{Verra}. 
Thus, by varying the plane $V$ and the abelian surface $P$, 
we can construct all unramified double covers $\pi : Y \to X$ of genus $3$ quartic plane curves.

Since $\mathfont{Jac}(Y) \sim \mathfont{Jac}(X) \times P$, the curve $Y$ will be supersingular if and only if 
both $P$ and $X$ are supersingular.
Thus we choose $P$ to be a supersingular p.p.\ abelian surface and 
search for a plane $V$ for which $X=K \cap V$ is also supersingular. 
As further explained in Rationale~\ref{R2}, it is reasonable to expect this to work as:
\begin{itemize}
\item  the supersingular locus in $\mathcal{A}_2$ has dimension $1$;
\item  the moduli space of planes in $\PP^3$ has dimension $3$;
\item and the codimension of the supersingular locus in $\mathcal{M}_3$ is $4$.
\end{itemize}

It is difficult to turn this into a rigorous argument due to the existence of singular examples as will be discussed in Section~\ref{ss:singularfamilies}.  But the idea is the basis for our computational search.

\subsection{Searching for Supersingular Curves}

To implement this idea as an algorithm, we express $P$ as the Jacobian of a smooth curve $Z$ of genus two.  
The projective curve $Z$ has an affine equation of the form $y_1^2=D(x_1)$
for a separable polynomial $D(x_1) =\sum_{i=0}^6 d_i x_1^i$ of degree $6$.
(This may require a change of variables so that the cover $Z \to \mathfont{P}^1_{x_1}$ is not branched at infinity).
By \cite[(3.1.8)]{CasselsFlynn}, a projective model of the
Kummer surface $K$ in $\mathfont{P}^3$
is the zero locus of the equation
\[\kappa(x, y, z, w) = K_2w^2+ K_1w + K_0, \quad \text{where}\]
\begin{eqnarray*}
K_2 & \colonequals & y^2-4xz, \\
K_1 & \colonequals & -2(2d_0x^3 + d_1x^2y + 2d_2x^2z + d_3xyz+2d_4xz^2+d_5yz^2 + 2d_6z^3), \text{\ and} \\
K_0 & \colonequals & (d_1^2-4d_0d_2)x^4 -4d_0d_3x^3y -2d_1d_3x^3z -4d_0d_4x^2y^2 \\
& & + 4(d_0d_5-d_1d_4)x^2yz + (d_3^2+2d_1d_5-4d_2d_4 - 4d_0d_6)x^2z^2 -4d_0d_5xy^3 \\
& & + 4(2d_0d_6 - d_1d_5)xy^2z + 4(d_1d_6 - d_2d_5)xyz^2 -2d_3d_5xz^3 -4d_0d_6y^4 \\
& & -4d_1d_6y^3z -4d_2d_6y^2z^2 -4d_3d_6yz^3 + (d_5^2-4d_4d_6)z^4.
\end{eqnarray*}
The singularities of $K$ are explicitly given in terms of the roots of $D(x)$ (see \cite[(3.1.14)]{CasselsFlynn} for the exact formulas).

We represent the plane $V$ in $\mathfont{P}^3$ by $v(x,y,z,w) = ax+by+cz + dw = 0$ for $[a,b,c,d] \in \mathfont{P}^3(k)$.

Using this description of $K$ and $V$, it is now feasible to search for instances where $P$ is supersingular and $X=K \cap V$ is supersingular using Magma \cite{magma}.  
For simplicity, we restrict the search to the case that $Z$ and $V$ are defined over the prime field $\mathfont{F}_p$.

\begin{proposition} \label{Pdata}
For each prime $p \equiv 1 \bmod{4}$ with $p< 100$,
consider the polynomial $D(x_1)$ and the linear polynomial $v=v(x,y,z,w)$ with coefficients in $\mathfont{F}_p$
in the row labeled by $p$ in Table~\ref{table:ss}.
Let $Z$ be the genus $2$ curve with affine equation $y_1^2=D(x_1)$.
Let $P = \mathfont{Jac}(Z)$, and let $K=P/[-1] \subset \PP^3$ be its Kummer surface.
Let $V \subset\PP^3$ be the plane $v(x,y,z,w)=0$, and let $X=K \cap V$.
Then 
\begin{enumerate}
\item $P$ is supersingular; 
\item $X$ is a supersingular smooth quartic plane curve; 
\item there is an unramified double cover $\pi : Y \to X$ whose Prym is $P$, 
and 
\item $Y$ is a supersingular smooth curve of genus $5$. 
\end{enumerate}

\begin{table}[htb]
\begin{tabular}{ |c|c|c| } 
 \hline
 \hspace{.5cm} $p$  \hspace{.5cm} & $D(x_1)$ & $v$ \\
 \hline 
 $5$ & $x_1^6 + x_1^5 + 2 $ & $ y + z + w  $ \\ 
 \hline 
 $13$ & $ 5x_1^6 + 5x_1^5 + 11x_1^4 + 4x_1^2 + x_1 + 4$ & $4x + y + 11z + w   $ \\ 
  \hline 
 $17$ & $ 15x_1^6 + 6x_1^5 + x_1^4 + 3x_1^3 + 3x_1^2 + 13x_1 + 3$ & $ 8x + y + 9z + w  $ \\ 
  \hline 
 $29$ & $ 21x_1^6 + 23x_1^5 + 6x_1^4 + 3x_1^3 + x_1^2 + 4x_1 + 17$ & $27x + 7y + 28z + w   $ \\ 
  \hline 
 $37$ & $ x_1^5 + 36$ & $ 6x + 6y + 4z + w  $ \\ 
  \hline 
 $41$ & $ 33x_1^6 + 33x_1^5 + 8x_1^4 + 21x_1^3 + 40x_1^2 + x_1 + 3 $ & $ 9x + 9y + 32z + w  $ \\ 
  \hline 
 $53$ & $  x_1^5 + 52$ & $ 6x + 4y + 8z + w  $ \\ 
  \hline 
 $61$ & $ 3x_1^6 + 32x_1^5 + 49x_1^4 + 11x_1^3 + 3x_1^2 + 30x_1 + 
    16 $ & $26y + 30z + w   $ \\ 
  \hline 
 $73$ & $ x_1^5 + 72$ & $29x + 23y + 44z + w   $ \\ 
  \hline 
 $89$ & $ x_1^6 + 28x_1^5 + 24x_1^4 + 57x_1^3 + 63x_1^2 + 11x_1 + 77$ & $ 7x + 15y + 47z + w  $ \\ 
  \hline 
 $97$ & $ 39x_1^6 + 26x_1^4 + 44x_1^3 + 7x_1^2 + 28x_1 + 52$ & $ 89x + 6y + 67z + w  $ \\ 
 \hline
\end{tabular}
\caption{Constructions of Supersingular Curves for $p < 100$, $p\equiv 1 \bmod{4}$} \label{table:ss}
\end{table}
\end{proposition}

\begin{proof}
The code we wrote to search for examples is available on github \cite{github}.  
Using the LMFDB \cite{lmfdb}, for each prime $p$, 
we obtain a list of affine equations $y_1^2=D(x_1)$ for genus two curves $Z$ 
such that $P=\Jac(Z)$ is in a supersingular isogeny class of p.p.\ abelian varieties over $\mathfont{F}_p$.  
Given one of these, we compute its Kummer surface $K$. 
Then we search through planes $V$ defined over $\mathfont{F}_p$, and not containing any singularity of $K$, 
for which $X=K \cap V$ is a smooth quartic plane curve. 
We check the latter by verifying that $X$ is reduced, irreducible, and has arithmetic genus three.  
For each such pair $Z$ and $V$, by Section~\ref{SspecifiedPrym}, 
there exists an unramified double cover $\pi : Y \to X$ whose Prym is $P$.

We then use Magma to find cases when $X$ has $p$-rank $0$, which is a necessary condition for $X$ to be supersingular.
The $p$-rank is $0$ when the stable rank of $M$ is $0$, where $M$ is
a matrix representation for the Cartier operator on $H^0(X, \Omega^1)$.
(Since the coefficients of $M$ are in $\mathfont{F}_p$, having stable rank $0$ is conveniently equivalent 
to the condition that $M^3$ is the zero matrix.)  
For the cases satisfying this restrictive condition (of codimension three, geometrically),
we use Magma again to compute the L-polynomial of $X$ over $\mathfont{F}_p$ to check whether $X$ is supersingular.
\end{proof}  

\begin{remark}
Our goal in Proposition~\ref{Pdata} is simply to provide convincing evidence for Conjecture~\ref{conjecture} when $p\equiv 1\bmod{4}$.
So we did not attempt to maximize the range of primes $p$ checked or greatly optimize our code.  
The process does not inherently scale well in any case, because the number of planes to consider for each 
supersingular genus $2$ curve $Z$ is $O(p^3)$.
For example, when $p=97$, the search took around one day on a single processor of a desktop computer; 
the search was successful for the first curve $Z$ on the list from the LMFDB, 
and it ran through about $7\%$ of the possible planes $V$ before finding a supersingular genus $3$ curve 
$X=V \cap (\Jac(Z)/[-1])$.

We also considered a second computational approach to this problem by fixing $Z$ and considering a parametric representation 
$v: A x + By + C z + D w=0$ of $V$.
The Hasse--Witt matrix for $X=K \cap V$ can be determined from the coefficients of $(v\kappa)^{p-1}$, see \cite[Proposition~4.3]{WINEprym}.  
Our plan was to use a Gr\"{o}bner basis computation to express the $p$-rank $0$ condition on $X$ in terms of the parameters $A,B,C,D$.
This would let us restrict our search to planes $V$ for which $X$ has $p$-rank $0$, greatly reducing the time spent searching.  However, the entries of the Hasse--Witt matrix, as polynomials in $A,B,C,D$, have degree $p-1$, and the entries of its third iterate  
have degree $p^3-1 =(p-1)(p^2+p+1)$.  A preliminary exploration showed this approach was not feasible.
\end{remark}

\begin{example}
Working over $\mathfont{F}_5$, we construct a supersingular curve of genus $5$ 
which differs from the one found by Re \cite[Theorem~13]{re2009}.
The curve $Z : y_1^2 = x_1^6+x_1^5+2$ is supersingular
\cite[\href{https://www.lmfdb.org/Variety/Abelian/Fq/2/5/a_af}{Abelian variety isogeny class $2.5.a\_af$ over $\FF_{5}$}]{lmfdb}.
The Kummer surface $K$ has equation
\[
0 = \kappa = 2xy^3 + 2y^4 + 3x^2yz + xy^2z + 2x^2z^2 + z^4 + 2x^3w +
         3yz^2w + z^3w + y^2w^2 + xzw^2.\]
Let $V$ be the plane $y+z+w =0$.
Using Magma, we compute that 
$X=K \cap V$ is a smooth quartic plane curve which is supersingular.  
As $X$ does not contain any of the singularities of $K$, the restriction of $\phi : P \to K$ above $X$ is an unramified double cover 
$\pi:Y \to X$, where $Y$ is a smooth supersingular curve of genus $5$.

For this curve $Z$, there are six choices of plane $V$ with $a,b,c,d \in \FF_5$ for which $X=K \cap V$ is supersingular.  
This large number may be because
$\mathfont{Jac}(Z)$ is a twist 
of the Jacobian of the Artin--Schreier curve $y_1^2=x_1^5-x_1$
\cite[\href{https://www.lmfdb.org/Variety/Abelian/Fq/2/5/a_af}{Abelian variety isogeny class $2.5.a\_ak$ over $\FF_{5}$}]{lmfdb}.

The curve $Z' : y_1^2= x_1^5 + 3x_1$ is also supersingular, but there are no planes $V : ax + by + cz + dw =0$ with $a,b,c,d \in \FF_5$ for which $V \cap (\Jac(Z')/[-1])$ is supersingular.
\end{example}


\begin{example}
Working over $\mathfont{F}_{37}$,
the curve $Z : y_1^2 = x_1^5 -1$ is supersingular \cite[\href{https://www.lmfdb.org/Variety/Abelian/Fq/2/37/a_a}{Abelian variety isogeny class $2.37.a\_a$ over $\FF_{37}$}]{lmfdb}.  Another equation for $Z$ is $y_1^2 = x_1-x_1^6$.
Using the latter equation, we compute that the Kummer surface $K$ has equation
\[
0 = \kappa = x^4 + 4y^3z - 4xyz^2 - 2x^2yw + 4z^3w + y^2w^2 - 4xzw^2 =0.
\]
Let $X = K \cap V$ where $V: 6x + 6y + 4z + w=0$. 
As $X$ does not contain any of the singularities of $K$, the restriction of $\phi : P \to K$ above $X$ is an unramified double cover 
$\pi:Y \to X$, where $Y$ is a smooth supersingular curve of genus $5$.
\end{example}


\section{Rationale for the conjecture} \label{Srationale}

We provide several rationales for Conjecture~\ref{conjecture} and explain difficulties posed by singular curves.

\subsection{First Rationale}

\begin{notation} \label{Nmoduli}
For $g \geq 1$, let $\sigma_g$ denote the supersingular Newton polygon of height $2g$.
Let $\cA_g$ denote the moduli space of principally polarized abelian varieties of dimension $g$,
and let $\cA_g[\sigma_g]$ denote its supersingular locus.
For $g \geq 2$, let $\cM_g$ denote the moduli space of smooth curves
of genus $g$, and let $\cM_g[\sigma_g]$ denote its supersingular locus.
Let $\tau_g: \cM_g \to \cA_g$ be the Torelli morphism, which takes the isomorphism class of a curve of genus $g$
to the isomorphism class of its Jacobian.
\end{notation}

Suppose $X$ is a smooth curve of genus $g$.  
If $\pi: Y \to X$ is an unramified double cover, then $Y$ is a smooth curve of genus $2g-1$ and
$\Jac(Y)$ is isogenous to $\Jac(X) \times P$, where the Prym $P$ of $\pi$ is a p.p.\ abelian variety of dimension $g-1$.
So $Y$ is supersingular if and only if both $X$ and $P$ are supersingular.
Consider the following:
\begin{equation} \label{Econdition1}
\text{Condition~1: } \mathrm{dim}(\cM_g[\sigma_g]) \geq \mathrm{codim}(\cA_{g-1}[\sigma_{g-1}], \cA_{g-1}).
\end{equation}

When Condition~1 is true, the search for a supersingular smooth curve of genus $2g-1$ 
which is an unramified double cover $\pi:Y \to X$ of a curve $X$ of genus $g$ 
is more likely to be successful because of the purity theorem of de Jong and Oort \cite[Theorem~4.1]{JO00}.
However,  Condition~1 cannot be true for many $g$: 
the right hand side of \eqref{Econdition1} equals $g(g-1)/2 - \lfloor {(g-1)^2/4} \rfloor$ by \cite[Theorem in Section~4.9]{LiOort};
while the left hand side is bounded above by $2g-3$ (the dimension of the $p$-rank $0$ stratum of $\cM_g$) by \cite[Theorem~2.3]{FVdG}.

\begin{rationale} \label{R1}
The first rationale for Conjecture~\ref{conjecture} is that Condition~1 is true when $g=3$.
Every irreducible component $\Gamma$ of $\cA_3[\sigma_3]$ has dimension $2$.
The image of $\tau_3: \cM_3 \to \cA_3$ is open and dense in $\cA_3$.
So every irreducible component $\Gamma$ of $\cM_3[\sigma_3]$ has dimension $2$.
Also $\mathrm{codim}(\cA_2[\sigma_2], \cA_2) = 2$.
Thus, by varying $X$ in $\Gamma$, we can hope to find an unramified double cover $\pi:Y \to X$ 
whose Prym is supersingular.
\end{rationale}

Condition~1 also holds for $g=2$, because $\mathrm{dim}(\cM_2[\sigma_2])=1$ and 
$\mathrm{codim}(\cA_{1}[\sigma_{1}], \cA_{1})=1$.  
Given $p$, one can also ask about supersingular curves of genus $3$ that are
unramified double covers of genus $2$ curves.
Surprisingly, we are not aware of any results on this so we include one here.

\begin{proposition} \label{Punramsmall}
If $p \equiv 3 \bmod 4$, then over $\overline{\mathfont{F}}_p$ 
there exists an unramified double cover $\pi: Y \to X$ of a genus $2$ curve such that
$Y$ is a supersingular smooth curve of genus $3$.
\end{proposition}

\begin{proof}
If $p \equiv 3 \bmod 4$, for $\beta \in k -\{0,1\}$, consider $X: y^2=x(x^2-1)(x^2-\beta)$ and 
the projection map $h_1: X \to \mathfont{P}^1$ onto the $x$-axis.
Let $E: y^2=x(x^2-1)$, with $h_2: E \to \mathfont{P}^1$ being the projection onto the $x$-axis.
Then the pullback of $h_2$ by $h_1$ is an unramified double cover $\pi:Y \to X$.

Also $Y \to \mathfont{P}^1$ is a Klein-four cover whose third intermediate quotient has genus $0$.
By \cite[Theorem~B]{kanirosen}, $\Jac(Y) \sim \Jac(X) \times \Jac(E)$.
If $p \equiv 3 \bmod 4$, then $E$ is supersingular (and independent of $\beta$).
By \cite[Propositions~1.9, 1.14]{iko}, there are (approximately $p/4$) choices of $\beta$ such that $X$ is supersingular.
Thus $Y$ is a smooth supersingular curve of genus $3$ for those choices of $\beta$, and $\pi:Y \to X$ satisfies the conditions in the 
statement.
\end{proof}

\subsection{Families of singular supersingular curves} \label{ss:singularfamilies}

Unfortunately, there are many \emph{singular} curves $X$ of genus $3$ having an unramified double cover $\pi:Y \to X$ 
such that $Y$ is supersingular.  
We explain how to construct positive dimensional families of these.
Geometrically, these families demonstrate that the subspace $\cM_3[\sigma_3]$ 
in Rationale~\ref{R1} does not intersect the boundary of $\cM_3$ in a dimensionally transverse way.

Let $\cM_g^{ct}$ denote the moduli space of stable curves of genus $g$ of compact type,
and let $\cM_g^{ct}[\sigma_g]$ denote its supersingular locus.  The Torelli morphism extends to $\tau_g: \cM_g^{ct} \to \cA_g$.
We first construct a family of supersingular singular curves of genus $3$, 
whose moduli points are in $\cM_3^{ct}[\sigma_3]$.

\begin{notation} \label{Nsingularcomponent}
Let $(E, \mathcal{O}_E)$ be a supersingular curve of genus $1$ with a marked point.
Let $(W, \eta)$ be a supersingular curve of genus $2$ with one marked point.
Let $X_s$ be the curve obtained by clutching $W$ and $E$ together at their marked points.
Then $X_s$ is a stable curve of compact type which has genus $3$ and is supersingular.
Also $\Jac(X_s) \cong \Jac(E) \times \Jac(W)$ by \cite[Ex.\ 9.2.8]{BLR}.

There are two dimensional families of such curves $X_s$, 
because there is a one-dimensional choice for $W$, and a one-dimensional choice for the point $\eta \in W$.
Note that $\Jac(W)$ does not depend on the choice of $\eta$.  
For all but the smallest primes, there is more than one choice for $E$, and 
for the one-dimensional family of curves $W$ \cite[Theorem in Section~4.9]{LiOort}.

More precisely, 
in the notation of \cite{knudsen2}, 
there is a clutching morphism $\kappa: \cM_{1;1} \times \cM_{2;1} \to \cM_3^{ct}$.
We restrict $\kappa$ to the supersingular locus $\kappa_{ss} : \cM_{1;1}[\sigma_1] \times \cM_{2;1}[\sigma_2] \to \cM_3^{ct}[\sigma_3]$.
Let $S$ be an irreducible component of the image of $\kappa_{ss}$.  
Then $\mathrm{dim}(S) = 2$ and $\mathrm{dim}(\tau_3(S))=1$.
\end{notation}

Given a singular genus $3$ supersingular curve $X_s$ as in Notation~\ref{Nsingularcomponent}, 
we now construct unramified double covers of $X_s$ that are supersingular (and still singular).

\begin{notation} \label{Nsingcomp}
Let $S$ be an irreducible component of the image of $\kappa_{ss}$ as in Notation~\ref{Nsingularcomponent}.
For $s \in S$, consider an unramified double cover $\pi_s: Y_s \to X_s$ as follows.
\begin{description}
\item[Case 1] Suppose the restriction of $\pi_s$ over $W$ 
is disconnected.  Then $Y_s$ is a singular curve of compact type, having 
three irreducible components $W', W'', E'$, where $W' \cong W'' \cong W$ and
where $E' \to E$ is an unramified double cover.  
So $Y_s$ is supersingular.
The Prym of $\pi_s$ is isogenous to $\mathrm{Jac}(W)$, thus varies with $W$.
Unfortunately, this provides a 1-dimension family of such covers; the number of these covers defined over $\mathfont{F}_q$ grows with $q=p^a$.

\item[Case 2] Suppose the restriction of $\pi_s$ over $E$ 
is disconnected.  Then $Y_s$ is a singular curve of compact type, 
having three irreducible components $\tilde{W}, E', E''$, where $E' \cong E'' \cong E$ and where $\tilde{W} \to W$ is an unramified 
double cover. The curve $Y_s$ is supersingular if and only if the Prym $P_2$ of $\tilde{W} \to W$ is supersingular.
In certain cases, e.g.\ when $p \equiv 3 \bmod 4$ using Proposition~\ref{Punramsmall}, we know that $P_2$  
is supersingular for some double cover $\tilde{W} \to W$.
Then the Prym of $\pi_s$ is isogenous to $P_2 \times E$ and thus is supersingular (and superspecial).
\end{description}
\end{notation}

We would like to understand which of the curves $X_s$ from Notation~\ref{Nsingularcomponent}
are in the closure of $\cM_3[\sigma_3]$ in $\cM_3^{ct}$, 
and which of the unramified double covers $\pi_s: Y_s \to X_s$ from Notation~\ref{Nsingcomp}
are in the closure of the supersingular locus of the moduli space of unramified double covers of smooth curves of genus $3$.  
This would help in resolving Conjecture~\ref{conjecture} for $p \equiv 1 \bmod 4$.

\subsection{Second Rationale}

\begin{notation} \label{NPrymvariety}
Let $\cR_g$ denote the moduli space of unramified double covers $\pi:Y \to X$ where $X$ is a smooth curve of genus $g$.
The forgetful morphism $\cR_g \to \cM_g$ takes the isomorphism class of $\pi$ to the isomorphism class of $X$; it
is finite and unramified.
The Prym morphism $\rho_g: \cR_g \to \cA_{g-1}$ takes the isomorphism class of $\pi$
to the isomorphism class of the Prym of $\pi$.
\end{notation}

Here is the second rationale for Conjecture~\ref{conjecture}.

\begin{rationale} \label{R2}
Every irreducible component $\Xi$ of $\cA_2[\sigma_2]$ has dimension $1$.
If $P$ is a p.p.\ abelian surface, consider
the fiber of the Prym map $\rho_3: \cR_3 \to \cA_2$ over the moduli point for $P$.
By \cite[Corollary~4.1]{Verra}, this fiber contains one component of dimension $3$, 
whose points represent quartic plane curves $X$ that have an unramified double cover 
$\pi:Y \to X$ whose Prym is $P$. 
Let $R$ be the pre-image $\rho_3^{-1}(\Xi)$.
Since $\cR_3 \to \cM_3$ is finite and unramified,
the dimension of the image of $R$ in $\cM_3$ is $4$.
In addition, $\mathrm{codim}(\cM_3[\sigma_3], \cM_3) = 4$.
Thus, by varying the moduli point of $P$ in $\Xi$, and the moduli point in the fiber of $\rho_3$ above it, 
we might expect to find (a finite number of) supersingular curves $X$, 
having an unramified double cover $\pi: Y \to X$ such that $Y$ is supersingular.
\end{rationale}

\begin{rationale} \label{Rheuristic}
We present additional information about Rationale~\ref{R2} using intersection theory in the tautological ring of $\cA_3$.
We thank Jeremy Feusi and Renzo Cavalieri for explaining some of these ideas to us.
For $g \geq 1$, let $\mathfont{E}_g\to \cA_g$ denote the Hodge bundle, which is the cotangent bundle of the zero-section of the universal
p.p.\ abelian variety of dimension $g$;
if $X$ is a curve of genus $g$, the sections of $\mathbb{E}_g$ over $\Jac(X)$ are the holomorphic $1$-forms on $X$.  
Consider the Chern classes $\lambda_i$ of $\mathfont{E}_g$ for $1 \leq i \leq g$.

The supersingular locus $\mathcal{A}_2[\sigma_2]$ has cycle class 
$\gamma_1 = f_1(p) \lambda_2$, where $f_1(p)=(p-1)(p^2-1)$ \cite[Example~12.2]{EVdG}.
Consider the pre-image $\rho_3^{-1}(\mathcal{A}_2[\sigma_2])$ in $\cR_3$.
Let $\gamma_1'$ be its image under the morphism $\mathcal{R}_3 \to \mathcal{M}_3 \to \mathcal{A}_3$, which takes 
$[\pi: Y \to X] \mapsto [X] \mapsto [\mathfont{Jac}(X)]$.
Using the fact that $\cR_3 \to \cM_3$ is finite and unramified of degree $63$, one can show that
$\gamma_1'$ has cycle class $63 f_1(p) \lambda_2$ in the tautological ring of $\cA_3$.

The supersingular locus $\mathcal{A}_3[\sigma_3]$ has cycle class 
$\gamma_2 = f_2(p) \lambda_1\lambda_3$, where $f_2=(p-1)^2(p^3-1)(p^4-1)$ \cite[Theorem~8.1]{VdGHara}.  
The intersection of $\gamma'_1$ and $\gamma_2$ is $N_p:=63f_1(p)f_2(p)\mathfont{deg}(\lambda_1 \lambda_2 \lambda_3)$.
By the Hirzebruch--Mumford proportionality theorem, 
$\mathfont{deg}(\lambda_1 \lambda_2 \lambda_3) = (1/8) \zeta(-1)\zeta(-3)\zeta(-5)$, where $\zeta(z)$ is the Riemann--zeta function.
So $N_p=f_1(p)f_2(p)/(2^{10} \cdot 3^2 \cdot 5)$, which has rate of growth $O(p^{12})$. 
This is promising, but inconclusive because of the excess intersection discussed in Notation~\ref{Nsingcomp}.
\end{rationale}

\subsection{A variation of the conjecture}

Let $X$ be a smooth curve of genus $g$.
Suppose $\pi': Y' \to X$ is a double cover branched at exactly two points.
Then $Y'$ is a smooth curve of genus $2g$ and
$\Jac(Y')$ is isogenous to $\Jac(X) \times P'$, where the Prym $P'$ of $\pi'$ is a p.p.\ abelian variety of dimension $g$.
So $Y'$ is supersingular if and only if both $X$ and $P'$ are supersingular.
Let $\cM_{g;2}$ denote the moduli space of smooth genus $g$ curves with two marked points.
Consider the following:
\begin{equation} \label{Econdition1}
\text{Condition~2: } \mathrm{dim}(\cM_{g;2}[\sigma_g]) \geq \mathrm{codim}(\cA_{g}[\sigma_{g}], \cA_{g}).
\end{equation}
When Condition~2 holds, 
searching for a supersingular curve of genus $2g$ 
which is a double cover $\pi:Y' \to X$ of a smooth curve $X$ of genus $g$ is more likely to be successful
because of \cite[Theorem~4.1]{JO00}.
Similarly, we see that Condition~2 cannot be true for many $g$.

Condition~2 is satisfied when $g=2$ because $\mathrm{dim}(\cM_{2;2}[\sigma_2])=3$ and 
$\mathrm{codim}(\cA_{2}[\sigma_{2}], \cA_{2})=2$.
In fact, for all odd $p$, the supersingular curves of genus $4$ found in \cite{khs20} 
are double covers of curves of genus $g=2$.
Here is an example where these supersingular curves can be written down easily.

\begin{example}
If $p \equiv 5 \bmod 6$, then over $\overline{\mathfont{F}}_p$, there is a genus $2$ curve $X$
with a double cover $\pi': Y' \to X$ branched at two points such that $Y'$ is a smooth supersingular curve of genus $4$.
\end{example}

\begin{proof}
For $\alpha \in k -\{0,1\}$, consider 
the smooth genus $2$ curve $X: y^2 = (x^3-1)(x^3 - \alpha)$ and the projection map $h_1: X \to \mathfont{P}^1$ onto the $x$-axis.
Let $E: y^2=x^3-1$ with $h_2: E \to \mathfont{P}^1$ being the projection onto the $x$-axis.
The pullback of $h_2$ by $h_1$ is a double cover $\pi':Y' \to X$ whose branch locus consists of the two pre-images of $\infty$ in $X$.

Also $Y' \to \mathfont{P}^1$ is a Klein-four cover whose third intermediate quotient is $E':y^2 = x^3-\alpha$.
By \cite[Theorem~B]{kanirosen}, $\Jac(Y') \sim \Jac(X) \times \Jac(E) \times \Jac(E')$.

If $p \equiv 5 \bmod 6$, then $E$ and $E'$ are supersingular elliptic curves (for any $\alpha$).
By \cite[Propositions~1.8, 1.14]{iko}, there are (approximately $p/3$) choices of $\alpha$ such that $X$ is supersingular. 
Thus $Y'$ is a smooth supersingular curve of genus $4$ for those choices of $\alpha$, 
and $\pi:Y' \to X$ satisfies the conditions in the statement.
\end{proof}

Condition~2 is also holds when $g=3$ because $\mathrm{dim}(\cM_{3;2}[\sigma_3])=4$ and 
$\mathrm{codim}(\cA_{3}[\sigma_{3}], \cA_{3})=4$.
Thus we include the following conjecture.

\begin{conj} \label{conjecture2}
For any odd prime $p$, 
there exists a smooth curve $X'$ of genus $3$ over $\overline{\FF}_p$
with a double cover $\pi' : Y' \to X'$ branched at two points such that
$Y'$ is a supersingular curve of genus $6$.
\end{conj}

\bibliographystyle{amsalpha}
\bibliography{supersingular5}

\end{document}